\documentclass[11pt]{amsart}

\usepackage{amsmath,amsfonts,amsthm,amsopn,cite,mathrsfs}
\usepackage{epsfig,verbatim,color}
\usepackage{subfigure}
\newcommand{\pana}{poly-analytic}

\def\cH{\mathcal{ H}}
\def\cB{\mathcal{ B}}

\def\cA{\mathcal{ A}}

\def\cO{\mathcal{ O}}

\newcommand{\fif}{if and only if}

\def\rd{\bR^d}
\def\cd{\bC^d}
\def\rdd{{\bR^{2d}}}

\def\lrd{L^2(\rd)}

\newcommand{\wid}{Wigner distribution}

\def\intrd{\int_{\rd}}

\def\R{\right)}
\def\l{\langle}
\def\r{\rangle}
\def\<{\left<}
\def\>{\right>}

\def\inv{^{-1}}

\def\mv1{M_v^1}

\newtheorem{lemma}{Lemma}[section]
\newtheorem{proposition}[lemma]{Proposition}
\newtheorem{theorem}[lemma]{Theorem}

\theoremstyle{definition}

\newtheorem{remark}{Remark}
\newtheorem{example}{Example}

\newtheorem{exmp}{Example}[section]

\newcommand{\stft}{short-time Fourier transform}

\def\Z{{\mathbb{Z}}}
\def\Q{{\mathbb{Q}}}
\def\R{{\mathbb{R}}}
\def\C{{\mathbb{C}}}

\def\ffi{{\varphi}}

\newcommand{\un}{\mathbf{1}}

\newcommand{\field}[1]{\mathbb{#1}}
\newcommand{\bR}{\field{R}}        
\newcommand{\bN}{\field{N}}        
\newcommand{\bZ}{\field{Z}}        
\newcommand{\bC}{\field{C}}        
\newcommand{\bQ}{\field{Q}}        
\newcommand{\bT}{\field{T}}        %

\renewcommand{\d}{\,\mathrm{d}}

\begin{document}
\begin{abstract}
We study the question under which conditions the zero set of a
(cross-) Wigner
distribution $W(f,g)$ or a \stft\ is empty. 
 This is the case when  both $f$ and $g$ are
generalized Gaussians, but we will construct  less obvious examples
consisting of exponential functions and their convolutions. The
results require elements from the theory of totally positive
functions, Bessel functions, and Hurwitz polynomials.
The question  of zero-free \wid s is also  related to Hudson's theorem for the
positivity of the Wigner distribution and to Hardy's uncertainty
principle.   We then construct a class of  step functions
$S$  so that  the \wid\ $W(f,\mathbf{1}_{(0,1)})$ always possesses a
zero  $f\in S \cap L^p$ for $p<\infty$, but may be zero-free for
$f\in S \cap L^\infty$.  The examples show
that  the question of zeros of the Wigner distribution may be quite
subtle and relate to several branches of analysis.  
\end{abstract}

\title[Zeros of the Wigner Distribution]{Zeros of the  Wigner Distribution and the Short-Time Fourier Transform}
\author{Karlheinz Gr\"ochenig}
\address{Faculty of Mathematics \\
University of Vienna \\
Nordbergstrasse 15 \\
A-1090 Vienna, Austria}
\email{karlheinz.groechenig@univie.ac.at}
\author{Philippe Jaming}
\address{Universit\'e de Bordeaux\\
Institut de Math\'ematiques de Bordeaux UMR 5251\\
Cours de la Lib\'eration\\
F 33405 Talence cedex, France}
\email{philippe.jaming@gmail.com}
\author{Eugenia Malinnikova}
\address{Department of Mathematical Sciences \\
  Norwegian University of Science and Technology\\
  Alfred Getz vei 1 \\
Trondheim, Norway}
\address{
School of Mathematics, Institute for Advanced Study \\
Einstein dr. 1, Princeton, NJ, 08540, USA}
\email{eugenia.malinnikova@ntnu.no}

\subjclass[2000]{}
\date{}
\maketitle

\section{Introduction}
The aim of this paper is to investigate the zero set of the Wigner distribution of two functions
$f,g\in L^2(\bR )$, 
\begin{equation}
  \label{eq:2}
  W(f,g)(z) = \intrd f(x+\frac{t}{2}) \bar{g}(x-\frac{t}{2}) e^{-2\pi
    i \l\xi , t\r} \, dt  \,  ,  \qquad z=(x,\xi ) \in \rdd.
\end{equation}
More precisely, we are investigating whether this zero set can be empty.
Results here directly extend to other phase-space representations to which
the Wigner transform is closely related. These
include the ambiguity function and the \stft\ $V_gf(z)=\langle f,\pi (z) g\rangle_{L^2(\mathbb{R}^d)}$,
$\pi(z)g(t)=e^{2i\pi\xi\cdot t}g(t-x)$
(see \eqref{eq:1}). 

The zero set of the \stft\ is important in the study of the  generalized
Berezin quantization and the injectivity of a general  Berezin transform. 
The thesis of D.\ Bayer~\cite{bayer10}, partially published in ~\cite{BG15}, contains
the  following result (under a mild condition on $f, g \in \lrd$): \emph{If $V_{g}f(z) \neq 0$ for all $z \in
  \rdd $, then the mapping $ T \to \cB T, \cB (z ) = \langle T \pi
  (z ) g, \pi (z)f \rangle$ is one-to-one on the space of bounded
  operators on $\lrd $.} The function $(z,w) \mapsto \langle T \pi
  (z ) g, \pi (w )f  \rangle$ may be interpreted as a special
  integral kernel associated to the operator $T$. Bayer's statement
  asserts that $T$ is uniquely determined by the \emph{diagonal} of
  this kernel. The assertation is relevant, because theorems of this
  type were known only in the context of complex analysis, see,
  e.g.~\cite[Cor.\ 1.70]{folland89}.  Zero-free
  Wigner distributions occur  prominently  in ~\cite{KLSW12,LS18} in a
  similar context. It is therefore natural to ask for  examples that satisfy Bayer's
  assumptions: 
	
	\medskip
	
	\noindent{\bf Question.} {\em Determine (all) pairs $(f,g)$ such that $W(f,g)$ does not vanish.}
	
	\medskip

Clearly, if $f(t) = g(t) = e^{-\pi t^2}$
  is a Gaussian, then $W(f,g)(x,\xi ) = 2^d e^{-2\pi (x^2+\xi ^2)} \neq 0$
  everywhere. More generally, if  $q,\tilde q$ are quadratic polynomials such
  that $f=e^{-q},g=e^{-\tilde q} \in \lrd $ (generalized Gaussian),  then the
  \wid\ $W(f,g)$  is also zero-free. For a while, we have been unable to produce
	other pairs $(f,g)$ for which $W(f,g)$  is zero-free which lead us to believe that
	such pairs might not exist. The aim of this paper is to show that our belief was false by providing
       several examples of  \wid s without
        zeros. In doing so, we explore several connections with other areas  of
        analysis. Although we cannot offer a coherent theory of
        zero-free \wid s, we feel that the
        above question is of interest in itself and  warrants a deeper
        analysis. 
Part of the appeal of this problem comes from the connection to
different branches of analysis. In our quest for zero-free \wid s we touched  the
following topics in analysis. 

   (i) \emph{Total positivity.}  
    Most of our examples of zero-free \wid s (or ambiguity functions) are obtained from the
   basic example of the    one-sided exponential function
   $e^{-at}\mathbf{1}_{(0,\infty )}(t)$ by convolution. These happen to
   be totally positive functions, as follows from the fundamental
   classification of totally positive functions by
   Schoenberg~\cite{sch51}. However, not all totally positive
   functions possess a zero-free ambiguity functions, as can be seen
   from the symmetric exponential function $e^{-a|x|}$. 

   (ii) \emph{Bessel functions and Hurwitz polynomials.} The \wid\  of the functions
   $f_n(t) = t^ne^{-t} \mathbf{1}_{(0,\infty )}(t)$ contains certain
   Bessel functions as a factor. Moreover, to verify that  the
   ambiguity function of $f_n$ is zero-free, one needs to know that
   certain polynomials have all their zeros  in the 
   left half-plane, i.e., whether they are Hurwitz polynomials. We
   answer this question by recourse to 
   the classical properties of Bessel functions.

   (iii) \emph{Hudson's theorem.}  The non-vanishing of a \wid\ seems inherently related to
Hudson's theorem which asserts that a \wid\ is non-negative, \fif\ both
$f$ and $g$ are generalized Gaussians~\cite{hudson74}. Indeed,
variations of the proof of Hudson's theorem yield the following 
statement: \emph{If $\mathrm{Re}\, W(f,g)$ does not have any zeros,
  then $f$ and $g$ must be generalized Gaussians.} Similar statement
holds for the imaginary part of $W(f,g)$. For all other functions, the real and
imaginary parts of the \wid\  must  change sign, so the zero sets contain some
hypersurfaces and the question is when these hypersurfaces do not
intersect. 

(iv) \emph{Poly-analytic functions. } We also consider the case when
both functions $f$ and $g$ are finite linear combinations of Hermite 
functions. We are  lead to investigate the zeros of
certain \pana\ functions. Results of Balk~\cite{balkbook} yield some hints about the
zero set in this case. However, this direction does not seem to lead
to new examples of zero-free Wigner distributions. We conjecture that
no such examples exist in the class of  finite linear combinations of
Hermite 
functions.

(v)\emph{ Convexity and almost
periodicity.} As a last class of functions we consider pairs $f,g$
where one of the functions, say $g$,  is a characteristic function of
an interval.  We conjecture  that for any choice of $f\in L^2(\R)$ the
Wigner distribution $W(f,g)$ has zeros. To support our conjecture we
study the following particular case. 
Let $g=\un _{(0,1)}$ and choose $f$ to be a step function with
discontinuities on $\bZ \cup \alpha \bZ $ for irrational $\alpha$. If
$f\in L^p$ for $p<\infty$, the \wid\ $W(f, \un _{(0,1)})$ must always  have a
zero. However, we will construct  a delicate  example of a step function
$f$ in $L^\infty $ for which  the corresponding \wid\
$W(f,\un_{(0,1)})$ does not have any zeros. This example shows that
the non-existence of zeros of the \wid\  may be  quite
subtle and may depend sensitively on integrability or smoothness
properties of the function classes considered. 
In this part we use convexity and almost
periodicity as tools.

The article is organized as follows: in Section~\ref{sec:properties}
we list some of the properties of the Wigner distribution  and the
ambiguity function. 
We  list examples we found for which the Wigner distribution is
zero-free  in Sections~\ref{sec:example} and~\ref{sec:pos}.
In Section~\ref{sec:hudson} we study the
connection of zero-free \wid s to Hudson's theorem. Open questions and
interesting connections of the problem to other areas of analysis are
discussed in Sections~\ref{sec:pos}-\ref{sec:step}. 
In Section~\ref{sec:polyanal} we consider $f$ and $g$ 
to be finite linear combinations of Hermite functions and relate the
corresponding \wid\ to the theory of \pana\ functions. Section~\ref{sec:step} is
devoted to the case when $g$ is the characteristic function of an interval.

\section{Some properties of the Wigner transform}\label{sec:properties}

The Wigner transform is closely related to two other transforms in time-frequency analysis.
The first one is the short-time Fourier transform 
\begin{equation}
  \label{eq:1}
  V_gf(x,\xi ) = \intrd f(t) \overline{g(t-x)} e^{-2\pi i \l\xi , t\r} \,\mathrm{d}t  \, .
\end{equation}
A simple computation shows that
\begin{equation}
  \label{eq:c1}
  W(f,g)(x,\xi ) = e^{4\pi i x\cdot \xi } V_{Ig}f(2x, 2\xi ) \, ,
\end{equation}
with $Ig(t) = g(-t)$.
The second transform is the ambiguity function which is a slighlty
more symmetric version of the short-time Fourier transform
$$
A(f,g)=\intrd
f\left(t+\frac{x}{2}\right)\overline{g\left(t-\frac{x}{2}\right)}e^{-2i\pi\l\xi
  ,t\r}\d t \, .
$$
These transforms are related by the formulas 
$$
A(f,g)=e^{i\pi\xi x}V(f,g)\quad\mathrm{and}\quad A(f,g)(x,\xi)=\frac{1}{2}W(f,Ig)(x/2,\xi/2).
$$
Thus a pair $(f,g)$ for which $A(f,g)$ does not vanish also provides a
pair $(f,g)$ for which the short-time Fourier transform 
and a pair $(f,Ig)$ for which the Wigner-transform do not vanish.

Let us now list the invariance properties of the ambiguity
function. To do so, we recall that, 
for $z=(x,\xi)\in\bR^{2d}$,  the
phase-space  shift of $f\in \lrd $ along $z$ is defined  by $\pi
(z)f(t) = e^{2\pi i   \xi \cdot t} f(t-x)$. To every symplectic 
matrix $\mathcal{A}\in \mathrm{Sp}\, (d,\bR)$ one can then associate a
unitary operator $\mu (\cA ) $ acting on $\lrd $ such that
\begin{equation}
  \label{eq:n1}
  \pi (\cA z) = \mu (\cA ) \pi (z) \mu (\cA )^* \qquad \forall z\in \rdd
  \, .
\end{equation}
Here $\mu (\cA )$ is determined only up to a phase factor and is called
a metaplectic operator. The existence of metaplectic operators follows
already from the Stone-von Neumann theorem~\cite[Thm.~1.50]{folland89} and is
all we need here.  For
the much deeper aspects of these operators and the subtle construction
of the metaplectic representation we refer to \cite{folland89} and the
 recent book~\cite{Goss11}. For Theorem~\ref{wigreal} we need
the following formulas about the interaction between the \wid\ and the
metaplectic operators.


\begin{lemma}
Let $f,g\in \lrd $,  $w=(a,b),w'=(a',b')\in \rdd $, and $\cA \in \mathrm{Sp}\,
  (d,\bR )$. Then for all $z\in \rdd $
  \begin{eqnarray}
        W(\pi (w)f,\pi (w')g)(z) &=& e^{2i\pi\l w-w',z\r+i\pi\l b+b',a-a'\r}W(f,g)\left(z-\frac{w+w'}{2}\right) \, \label{eq:n2} \\
    W(\mu (\cA)f,\mu (\cA)g)(z) &=& W(f,g)(\cA ^*z) \, . \label{eq:n3}
  \end{eqnarray}
\end{lemma}

See 
Prop.~(1.92) for \eqref{eq:n2} (with slightly different
normalizations), and Prop.~(4.28) for \eqref{eq:n3}, or the general references~\cite{Goss11,book}. 

As a consequence, if $W(f,g)$ does not vanish, so does $W(\pi (w)f,\pi (w')g)$ and $W(\mu (\cA)f,\mu (\cA)g)$.
Note that when $\cA=J=\begin{pmatrix}0&-I\\I&0\end{pmatrix}$, $\mu
(\cA)f$ is the Fourier transform of $f$. Thus every example of a
zero-free \wid\ yields a whole class of related examples.

We will use a further property of the ambiguity function. For $\xi\in\bR$, let us write $M_\omega f(t)=\pi(0,\omega)f(t)=e^{-2i\pi\omega t}f(t)$.
Then
$$
A(f,g)(x,\xi)=\bigl(M_{-\xi/2}f\bigr)*\bigl(M_{\xi/2}g^*\bigr)(x)
$$
where $g^*(t)=\overline{g(-t)}$. As a consequence
\begin{align}
A(f_1*f_2,g_1*g_2)(x,\xi)&=\bigl(A(f_1,g_1)(\cdot,\xi)*_1A(f_2,g_2)(\cdot,\xi)\bigr)(x)
  \notag \\
&:=\int_\R A(f_1,g_1)(t,\xi)A(f_2,g_2)(x-t,\xi)\,\mathrm{d}t. \label{eq:conv}
\end{align}

Finally we cite the version of Hardy's uncertainty principle for the
\wid\ taken from ~\cite{GZ01}. 
\begin{lemma}\label{hardy}
  (i) If $|W(f,f)(z)| \leq C e^{-2\pi |z|^2}  = C' W(h_0,h_0)(z)$ for all
  $z\in \rdd$, then $f = c h_0$.

  (ii) Let $\gamma \in \lrd $ be a generalized Gaussian and $\psi \in
  \lrd $. If $ W(\psi ,\psi )(z) \leq W(\gamma, \gamma )(z)$ for all
  $z\in \rdd$, then $\psi = c \gamma $.
\end{lemma}

\begin{proof}
  (i) is Hardy's uncertainty principle for the \wid\ as proved in  \cite{GZ01}
  (Lemma 3.3 and Remark 3.4). 

  (ii) follows from (i). We recall  that every generalized
Gaussian $\gamma= e^{Q(x)}\in L^2(\bR ^d)$ with a
quadratic polynomial $Q$ can be written in the form $\gamma= \pi (w) \mu
(\cA ) h_0$, where $h_0(x) = 2^{d/4} e^{\pi x\cdot x} $ is the
normalized Gaussian, see e.g., \cite{folland89}, Prop.~(4.73). 
  Let $\psi _0\in \lrd $ such that
  $\psi = \pi (w) \mu (\cA) \psi _0$. Using \eqref{eq:n2} and
  \eqref{eq:n3}, we obtain
  \begin{align*}
    |W(\psi _0,\psi _0)(\cA ^*(z-w)) | &=   | W(\psi ,\psi )(z)| \\
    &\leq C W(\gamma,\gamma) (z) =   C'  W(h_0,h _0)(\cA ^*(z-w)) \qquad
      z\in \rdd \, .
  \end{align*}
  By (i) this implies that $\psi _0 = c h_0$ and consequently $ \psi =
  c \gamma $. 
\end{proof}

\section{Examples}\label{sec:example}

In this section we give several  examples of pairs of functions for
which the ambiguity function  do not vanish.

\begin{example}\label{ex:gauss}
The first example is the Gaussian function. Write
$\gamma_a(t)=e^{-a\pi t^2}$ with $\mathrm{Re} \, (a)>0$. Then a direct, well known
computation
shows that, 
if $\mathrm{Re} \, (a),\mathrm{Re} \, (b)>0$,
\begin{equation}
\label{eq:ambgaus}
A(\gamma_a,\gamma_b)=(a+b)^{-1/2}\exp \Big( -\pi \frac{abx^2+\xi^2}{a+b}+i\frac{a-b}{a+b}x\xi\Big)
\end{equation}
so that $A(\gamma_a,\gamma_b)$ does not vanish. 

Using tensorisation and invariance properties of the Wigner distribution,
it follows that, if $f,g\in L^2(\bR ^d)$ are generalized Gaussians $f= e^{Q}$ (for
some quadratic polynomial), then $A(f,g)$ does not vanish. 
\end{example}

A second family of examples is given by the one-sided exponential. 

\begin{example}\label{ex:cauchy}
For $a>0$, let $\eta_a(t)=e^{-at}\mathbf{1}_{(0,+\infty)}(t)$. Then, for $a,b>0$
\begin{equation}
\label{eq:ambexp}
A(\eta_a,\eta_b)(x,\xi)=\exp\left(-\frac{a-b}{2}x-\frac{a+b}{2}|x|\right)\frac{e^{-i\pi\xi|x|}}{a+b+2i\pi\xi}
=\eta_{a,b}(x)\frac{e^{-i\pi\xi|x|}}{a+b+2i\pi\xi}
\end{equation}
where $\eta_{a,b}(x)= (b-a)\inv \, \eta _a \ast I\eta _b(x) =
\displaystyle\begin{cases}e^{-ax}&\mbox{when }x\geq 0\\
  e^{bx}&\mbox{when }x<0\end{cases}$. See also~\cite{janssen96}. 
In particular, 
$$
A(\eta_a,\eta_a)(x,\xi)=\frac{e^{-(a+i\pi\xi)|x|}}{2(a+i\pi\xi)} 
$$
does not have any zeros. 
Since $A(If,Ig)(x,\xi ) = A(f,g)(-x,-\xi )$, the ambiguity function of
the one-sided exponential $I\eta _b, b>0$, is
\begin{equation}
  \label{eq:n3a}
A(I\eta_b,I\eta_b)(x,\xi)=\frac{e^{-(b-i\pi\xi)|x|}}{2(b-i\pi\xi)} \, ,
\end{equation}
and is also zero-free. 

However, the ambiguity function $A(\eta_a,I\eta_b)(x,\xi)=0$ for
$x>0$, because in this case  $t\mapsto\eta_a(t+x/2)$ and $t\mapsto
I\eta_b(t-x/2)$ have disjoint supports. 
\end{example}

\begin{example}\label{ex:exp}
 Next we compute the ambiguity function of convolutions
  of exponentials of the form   $f=\eta_{a}*\eta_{b}$ or $f=\eta_{a}\ast I\eta_{b}$ 
 for $a,b>0$. For a compact formula we write $\eta _{-a} = I\eta _a$. 
  \begin{lemma}
    Let $a,b >0, a\not=b$. Then
    \begin{equation}
      \label{eq:n1a}
      A(\eta _a \ast \eta _{\pm b} , \eta _a \ast \eta _{\pm b}
      )(x,\xi ) = 
      \frac{1}{2(b^2-a^2) + 4\pi i (\pm b-a)\xi } \Big(\frac{e^{-(a+\pi i \xi
          |x|)}}{a+\pi i \xi} -\frac{e^{-(b\pm \pi i \xi )
          |x|}}{b \pm \pi i \xi} \Big) \, .
    \end{equation}
Furthermore, $A(\eta _a \ast \eta _{\pm b} , \eta _a \ast
    \eta _{\pm b})$ has no zeros. 
  \end{lemma}

  \begin{proof}
    To facilitate calculations, we set $\psi _u (x) = e^{-u|x|}$ for
    $u\in \bC $. Given $a,b>0$ set $u=a+\pi i \xi , v=b+\pi i \xi
    $. In this notation, \eqref{eq:ambexp} with $a=b$ becomes
$$
A(\eta _a, \eta _a)(x,\xi) = \frac{1}{2u} \, \psi _u(x) \qquad \text{
  and } \qquad A(\eta _{-b}, \eta _{-b})(x,\xi) =
\frac{1}{2\overline{v}} \, \psi _{\overline{v}}(x) \, .
$$
To find the ambiguity function of $f=\eta _a \ast \eta _b$, we use
\eqref{eq:conv} and find that 
\begin{equation}
  \label{eq:n2a}
A(f,f)(x,\xi ) =\frac{\psi_{u}\ast \psi _{v}(x)}{4uv} \, .  
\end{equation}
Then  for $x\geq0$ we obtain 
\begin{eqnarray*}
\psi_u*\psi_v(x)&=&\int_{-\infty}^0 e^{ut-v(x-t)}\,\mathrm{d}t
+\int_0^xe^{-ut-v(x-t)}\,\mathrm{d}t+\int_x^{+\infty}e^{-ut+v(x-t)}\,\mathrm{d}t\\
&=&\frac{e^{-vx}}{u+v}+\frac{e^{-ux}-e^{-vx}}{v-u}
    +\frac{e^{-ux}}{u+v}\\
&=& e^{-ux}\Big(\frac{1}{v+u} + \frac{1}{v-u}\Big) +
    e^{-vx}\Big(\frac{1}{v+u} - \frac{1}{v-u}\Big) \\
  &=& \frac{2}{v^2-u^2} \Big(ve^{-ux} -ue^{-vx}\Big) \, .
\end{eqnarray*}
Since $\psi _u$ is even, so is $\psi _u \ast \psi _v$, and then \eqref{eq:n2a} yields
\begin{equation}
  \label{eq:n4}
A(f,f)(x,\xi) = \frac{1}{2(v^2-u^2)} \Big( \frac{e^{-u|x|}}{u}
-\frac{e^{-v|x|}}{v} \Big) \, , 
\end{equation}
from which \eqref{eq:n1a} follows by substituting the values for $u$ and
$v$.

For the case $f=\eta _a \ast \eta _{-b}$, according to  \eqref{eq:n3a}
we only have to  replace $v$ by
$\overline{v}$ in the above derivation, whence \eqref{eq:n1a} follows
for $f$.

Now assume that $A(\eta _a \ast \eta _b, \eta _a \ast \eta _b)(x,\xi )
= 0$. Since $a\neq b$, the denominator $v^2-u^2 \neq 0$ and
\eqref{eq:n4} implies  that
\begin{equation}
  \label{eq:n5}
  \frac{e^{-u|x|}}{u} = \frac{e^{-v|x|}}{v} \, , 
\end{equation}
or
$$
e^{(b-a)|x|} = e^{(v-u)|x|} = \frac{u}{v} = \frac{a+\pi i \xi}{b+\pi i
  \xi} \, .
$$
For $b>a$, then left-hand side is greater than $1$, whereas the
modulus of the right-hand side is less than $1$, which is
impossible. Thus $A(\eta _a \ast \eta _b, \eta _a \ast \eta _b) $ does 
not have any zeros. Likewise for the case $b<a$. 

Similarly, if $A(\eta _a \ast \eta _{-b}, \eta _a \ast \eta _{-b})(x,\xi )
= 0$, then 
\begin{equation}
  \frac{e^{-u|x|}}{u} = \frac{e^{-\overline{v}|x|}}{\overline{v}} \, , 
\end{equation}
or
\begin{equation}
 \label{eq:n5b}
e^{(b-a - 2\pi i \xi) |x|} = e^{(\overline{v}-u)|x|} = \frac{u}{\overline{v}} = \frac{a+\pi i \xi}{b-\pi i
  \xi} \, .  
\end{equation}
 Again for $b>a$, the modulus of the left-hand side is greater than $1$, whereas the
modulus of the right-hand side is less than $1$, which is not
possible. Thus $A(\eta _a \ast \eta _{-b}, \eta _a \ast \eta _{-b}) $ does
not have any zeros. 
      \end{proof}

      The case $a=b$ and $\eta _a \ast \eta _{ a} = t e^{-at}
      \mathbf{1}_{(0,\infty )} $ is treated below. The case of  $\eta _a
      \ast \eta _{ -a} = e^{-a|x|}$ is particularly interesting, as
      its Fourier transform is the Poisson kernel. In this case
      \eqref{eq:n1a} holds for $\xi \neq 0$ and requires some easy 
      modification for $\xi = 0$. If $\xi \neq 0$, then \eqref{eq:n5b}
      turns into 
$$
e^{ - 2\pi i \xi |x|} = e^{(\overline{v}-u)|x|} = \frac{u}{\overline{v}} = \frac{a+\pi i \xi}{a-\pi i
  \xi} \, .  
$$
This equation is solvable, and therefore the ambiguity function of
$e^{-a|x|}$ must have zeros. In Theorem~\ref{tm-even} we will give a
different explanation for this fact.

\end{example}
It is also possible to derive a formula for the ambiguity function of
higher convolutions of the form $\eta _{a_1} \ast \eta _{a_2} \ast
\dots \ast \eta _{a_n}$ for distinct values of  $a_j> 0$. These
expressions are much more complicated,  and the analysis of 
their zeros requires a separate investigation. 

\begin{example}\label{ex:conv}
 Let $a>0$ and $f(t)=\eta_a*\eta_a(t)=t\eta_a(t)$. A direct computation shows that
$$
A(f,\eta_a)(x,\xi)=\frac{x\mathbf{1}_{(0,+\infty)}(x)e^{-a(1+i\pi\xi)x}}{2(a+i\pi\xi)}
+\frac{e^{-a(1+i\pi\xi)|x|}}{4(a+i\pi\xi)^2}.
$$
Obviously, $A(f,\eta_a)(x,\xi)$ can not vanish if $x\leq 0$,  while if
$x\geq 0$, it is enough to notice that $2x(a+i\pi\xi)+1$ 
can not vanish.
\end{example}

\begin{example}\label{ex:Gamma}
	 Let $\gamma_{a,b}(t)=\exp(at-be^t)$, where $a,b>0$, then $\gamma_{a,b}\in L^2(\R)$ and we claim that $A(\gamma_{a,b},\gamma_{a,b})$ does not vanish.
A straightforward computation  (with the substitution $s=2b\cosh
\tfrac{x}{2} e^t$) shows that
\begin{multline*}
A(\gamma_{a,b},\gamma_{a,b})(x,\xi)=\left(2b\cosh\frac{x}{2}\right)^{-2a+2\pi i\xi}
\int_0^\infty s^{2a-1-2\pi i \xi}e^{-s}ds=\\
\left(2b\cosh \frac{x}{2}\right)^{-2a+2\pi i\xi}\Gamma(2a-2\pi i\xi).
\end{multline*}
The  gamma function is well defined in the right half-plane and does not have zeros, so $A(\gamma_{a,b},\gamma_{a,b})$ never vanishes. A similar computation shows that 
\[A(\gamma_{a,b},\gamma_{c,d})=e^{(a-c)x/2}(be^{x/2}+de^{-x/2})^{-(a+c)+2\pi\xi}\Gamma(a+c-2\pi\xi)\]
and it has no zeros either. 
\end{example}

One can now use the invariance properties of the ambiguity function to
obtain more examples. For instance, using the Fourier transform, 
one obtains  the following examples. 

\begin{example}
Let us now write, for $a\not=0$, $c_a(t)=\frac{1}{a+2i\pi t}$ and note that $c_a$ is the Fourier transform of
$\eta_a$ when $a>0$. 
It follows that 
$A(c_a,c_b)$ does not vanish for $a,b >0$. 
 Furthermore, except for the Poisson kernel
$\frac{1}{a^2+4\pi ^2
  t^2}$ all rational functions of the form $c_a(t)
c_b(\pm t) = \frac{1}{a+2i\pi t}\frac{1}{b\pm 2i\pi t}$ possess a
non-vanishing ambiguity function, because  they are the Fourier
transform of $\eta _a \ast \eta _b$ or of $\eta _a \ast I\eta _b$.  
\end{example}

\textbf{Outlook and context.}  All functions in 
Examples \ref{ex:cauchy}-\ref{ex:conv} are
special cases of \emph{totally positive} functions. A function $g:\bR \to
\bR$ is called totally positive (or a Polya frequency function), if
for every choice of real numbers $x_1 < x_2 < \dots  <  x_n$ and
$y_1< y_2 < \dots  <  y_n$ with  $n\in \bN $ the inequality
$$
\det \big( g(x_j - y_k)\big)_{j,k=1, \dots , n} \geq 0 \, .
$$
According to a fundamental
characterization of Schoenberg~\cite{sch51} the Fourier
transform of a  totally
positive function in $L^1(\bR )$ can be factored as follows:
\begin{equation}
  \label{eq:1a}
\hat{g} (\xi ) = e^{-\gamma \pi \xi ^2} e^{2\pi i \delta \xi } \prod _{j=1}^\infty (1+2\pi i
\delta _j \xi )\inv e^{2\pi i \delta _j\xi }  \, ,
\end{equation}
with   $\delta , \delta _j\in \bR $, $0\leq  \gamma$,   and $ \sum _{j=1}^\infty \delta _j^2 < \infty $ and
$\delta \in \bR $. In this section we have analyzed the ambiguity
function of the simplest totally positive functions, corresponding to
the factorization $\hat{g}(\xi ) =  \prod _{j=1}^L (1+2\pi i
\delta _j \xi )\inv $ for $L=1$ and $L=2$. These are just the
functions $\eta _a$,  $\eta _a \ast \eta _b$, $\eta _a \ast I\eta
_b$.

At this time the  role of  totally positive
functions in the classification of zero-free ambiguity
functions remains unclear. Among all explicit examples considered,
only  the symmetric exponential $e^{-a|x|}$  has an ambiguity
function with zeros. More generally, the ambiguity function of  every
\emph{even} totally positive  function has a zero. See also
Theorem~\ref{tm-even} below.

\section{Totally positive functions and Hurwitz polynomials}
\label{sec:pos}
In this section we study a different class of totally positive
functions and its relation to Bessel functions and Hurwitz
polynomials. 
Let \[f_n(t) = t^n e^{-t} \mathbf{1}_{(0,\infty )}(t) = t^n \eta _1(t).\] This is a totally
positive function with Fourier transform $\hat{f}(\xi ) = (1+2\pi i
\xi )^{-n}$ ~\cite{sch51} and can also be written as $n! \, \eta _1
\ast \dots \ast \eta _1 $ ($n$ times).
Its ambiguity function can be explicitly calculated as follows. Let
$A_n$ be the polynomial of degree $n$ defined as follows:
\begin{equation}
  \label{eq:tp1}
A_n(z) = \sum _{k=0}^n \binom{n}{k} \, (n+k)! \,  z^{n-k}  
\end{equation}
\begin{lemma}
The ambiguity function of  $f_n(t) = t^n e^{-t} \mathbf{1}_{(0,\infty )}(t)$
is given as 
\begin{equation}
  \label{eq:tp2}
  A(f_n,f_n)(x,\xi ) = e^{-|x|(1+i\pi \xi )} \frac{1}{(2+2\pi i \xi
    )^{2n+1}} A_n\big( |x| (2+2\pi i \xi )\big) \, .
\end{equation}
  \end{lemma}

  \begin{proof}
    We rewrite the ambiguity function as follows:
    \begin{align*}
      A(f_n,f_n)(x,\xi ) &= \int _{\bR} \big(t+\tfrac{x}{2}\big)^n e^{-(t+x/2)}
      \mathbf{1}_{\bR ^+}(t+\tfrac{x}{2}) (t-\tfrac{x}{2}\big)^n e^{-(t-x/2)}
    \mathbf{1}_{\bR ^+}(t-\tfrac{x}{2}) \, e^{-2\pi i \xi t} \,  \d t  \\
      &= \int _{|x|/2} ^\infty \big(t+\tfrac{x}{2}\big)^n 
 (t-\tfrac{x}{2}\big)^n e^{-(2 + 2\pi i \xi ) t} \,  \d t \, , 
    \end{align*}
    since we need both $t>x/2$ and $t>-x/2$. For $x\neq 0$, we use  the change of variables
    $$
    u=\frac{t}{|x|} - \frac{1}{2}
    $$
    (so that $t= |x|u +|x|/2$) to rewrite   the integral as 
    \begin{align*}\label{eq:An}
A(f_n,f_n)(x,\xi ) &= |x|^{2n+1} \, \int _0^\infty u^n (u+1)^n
                 e^{-(|x|u+|x|/2) (2+2\pi i \xi )} \, \d u \\
      &= e^{-|x|(1+\pi i \xi )} |x|^{2n+1} \int _0^\infty u^n (u+1)^n
        e^{-u |x| (2+2\pi i \xi )} \, \d u \, .
    \end{align*}
    The integral is just the Laplace transform of the polynomial 
		\[u^n
    (u+1)^n = \sum _{k=0}^n \binom{n}{k} \, u^{n+k}\] at $\zeta  = |x|
    (2+2\pi i \xi )$.
    Since $\mathcal{L}\{u^k\}(z) = \int _0^\infty u^k e^{-uz}\, du = k!
    z^{-k-1}$, we obtain
    $$
    \int _0^\infty u^n (u+1)^n  e^{-u z} \, du \, =\sum _{k=0}^n
    \binom{n}{k}\,  (n+k)! \,  z^{-n-k-1} = z^{-2n-1} A_n (z)  \, .
    $$
    For $x=0$, the ambiguity function is
    $$
    A(f_n,f_n)(0,\xi ) = \int _0^\infty t^{2n} e^{-t(2+2\pi i \xi)} \d
    t = \mathcal{L}\{u^{2k}\} (2+2\pi i \xi) = (2n)! \, (2+2\pi i \xi
    )^{-2n-1} \, .
    $$
 With $\zeta  = |x|     (2+2\pi i \xi )$ and  $|x|^{2n+1} \zeta ^{-2n -1} =  (2+2\pi i \xi
    )^{-2n-1}$, the final formula for the ambiguity function on $\bR
    ^2$  is
    \begin{align*}
      A(f_n,f_n)(x,\xi ) &= e^{-|x|(1+i\pi \xi )} \frac{1}{(2+2\pi i \xi
    )^{2n+1}} A_n\big( |x| (2+2\pi i \xi )\big) \, .
    \end{align*}
  \end{proof}

  To show that the ambiguity function of $f$ does not vanish, we need to
know that $A_n$ has no zeros on the right half-plane. A polynomial
whose roots all have negative real part, is called a  Hurwitz
polynomial or a  stable polynomial. For small $n$ one can check the
stability of $A_n$ directly. For larger $n$,  one could check the
stability of $A_n$ in principle with the Routh-Hurwitz criterion \cite[pages 225-230]{Gant}. For
this purpose, we associate 
to a polynomial $p(t)=a_0t^n+a_1t^{n-1}+\cdots+a_n$, $a_0>0$
 the $n\times n$ Hurwitz matrix
$$
\cH=(a_{2j-i})_{1\leq i,j\leq n}=\begin{pmatrix}a_1&a_3&a_5&a_7&\cdots&0\\
a_0&a_2&a_4&a_6&\cdots&0\\
0&a_1&a_3&a_5&\cdots&0\\
0&a_0&a_2&a_4&\cdots&0\\
\vdots&\vdots&\vdots&\vdots&&\vdots\\
0&0&0&0&\cdots&a_n
\end{pmatrix}
$$
(with  the convention $a_j=0$ if $j<0$ or $j>n$). Then $p$ is a Hurwitz polynomial if and only if
all principal minors of $\cH$ are positive~\cite[pages 225-230]{Gant}.
This criterion implies that the condition $a_ja_{j+1} -
a_{j-1}a_{j+2}>0$ is necessary for the stability of $p$. On the other
hand, there exists an optimal value $\gamma>0$, $\gamma \approx
2.1479$, such that $a_ja_{j+1} > \gamma
a_{j-1}a_{j+2}$ is sufficient for $p$ to be stable~\cite{KV08}. One can check that
the polynomials $A_n$ satisfy the necessary condition, but fail to
satisfy the sufficient condition.  The decisive hint comes from
experimenting with 
Mathematica. It turns out that
\begin{equation}
  \label{eq:n6}
  A_n(z) = \pi ^{-1/2} n! e^{z/2} z^{n+1/2} K_{n+1/2}\big(\frac{z}{2}\big) \, ,
\end{equation}
where $K_\nu , \nu \in \bR $ is the Macdonald function or the  modified Bessel function of the
second kind~\cite{encycl,macdonald99,watson95}. For $\nu = n+1/2$ it is the Laurent polynomial
$$
K_{n+1/2}(z)  = \Big(\frac{\pi}{2z}\Big)^{1/2} e^{-z} \sum _{k=0}^n
\frac{(n+k)!}{k! (n-k)! } (2z)^{-k} \, ,
$$
whence \eqref{eq:n6} follows immediately. One can also perform the change of variables $s=2t/|x|$ in the first computation of $A(f_n,f_n)$ and directly obtain that
\begin{align*}
      A(f_n,f_n)(x,\xi ) &= \int _{|x|/2} ^\infty \big(t+\tfrac{x}{2}\big)^n 
 (t-\tfrac{x}{2}\big)^n e^{-(2 + 2\pi i \xi ) t} \,  \d t\\
&=
\left(\frac{|x|}{2}\right)^{2n+1}\int_1^\infty(s^2-1)^ne^{-(1+\pi i \xi)|x|s}\, \d s\\&=\frac{n!}{\sqrt{\pi}}\left(\frac{|x|}{(1+\pi i \xi)}\right)^{n+1/2}K_{n+1/2}(|x|(1+\pi i \xi))
\end{align*}
 Furthermore, for arbitrary
$\nu \geq 0$, $K_\nu $ has roots only when $\mathrm{Re}\, z <0$. In
view of \eqref{eq:n6} $A_n$ is therefore a Hurwitz polynomial~\cite{watson95}. 

By combining the accumulated knowledge about Bessel functions, we have
therefore proved the following result.

\begin{theorem} 
The polynomial $A_n$ is a Hurwitz polynomial.  Thus the  ambiguity function of $f_n(t) = t^n e^{-t}
  \mathbf{1}_{(0,\infty)}(t)$ does not have any zeros.  
\end{theorem}

Using the invariance of the problem under metaplectic operators, we
see that  the function $c_a^n (t) = \frac{1}{(a+2i\pi t)^n}$ has a
zero-free  ambiguity function. This follows  because 
the Fourier transform of $x^n e^{-ax}\mathbf{1}_{(0,\infty)}(x) = \eta
_a \ast \dots \ast  \eta _a$. 

\begin{remark}
{\rm Martin Ehler proposed a different proof that $A_n$ is a Hurwitz
  polynomial. He observed that  $A_n$ can be interpreted as the  hypergeometric
  function $_1F_1 (−n, −2n; z)$, whose zeros are known to lie in the
  left halfplane.} 
\end{remark}

\section{Connection to Hudson's Theorem}\label{sec:hudson}

We next discuss the connection of non-vanishing \wid s to Hudson's
theorem~\cite{hudson74}. This theorem characterizes all non-negative  \wid s as
follows:

\begin{theorem}[Hudson-Lieb]\label{hud}
 Let  $f,g\in \lrd $. Then  $W(f,g) \geq  0$, \fif\  $f= c g$ for
 $c\geq 0$ and $g$ is a generalized Gaussian. In this case
 $W(f,g)(x,\xi ) >0$ for all $x,\xi \in \bR $. 
\end{theorem}

As the connection to Hudson's theorem is not accidental, it is
instructive to review its proof. For full details, we refer to 
 \cite{hudson74,janssen84,book}, the
cited bilinear version is due to Lieb~\cite{lieb90}. 

One first uses the convolution relation $W(f,f) \ast W(g,g) =
|V_gf|^2$ for the \wid . Using  the Gaussian window $h_0(t) = 2^{d/4}
e^{-\pi t\cdot t}$ with \wid\ $W(h_0,h_0)(z) = 2^d e^{-2\pi |z|^2}$,
 we may   translate the statement about the zeros of the
 \wid\ into a statement
about entire functions in the Bargmann-Fock space.
The Bargmann transform of $f$ is defined to be 
\begin{equation}
  \label{eq:7}
  Bf(z) = 2^{d/4} e^{-\pi z^2/2} \intrd f(t) e^{-\pi  t \cdot  t} e^{2\pi
     t\cdot  z } \d t  \qquad z\in \cd \, ,
\end{equation}
and it is  connected to the \stft\ with respect to the Gaussian $h_0$
via 
\begin{equation}
  \label{eq:7b}
  V_{h_0}f(\bar z ) = e^{\pi i  x\cdot \xi } Bf(z) e^{-\pi |z|^2/2} \, .
\end{equation}
Here we identify the point $(x,\xi )\in \rdd $ with $z=x+i\xi \in \cd
$. If $W(f,f) \geq 0$, then $W(f,f) \ast W(h_0,h_0) (z) >0$ for all $z\in 
\rdd$ and therefore  
$$
|V_{h_0}f(\bar z)|^2 = |Bf(z)|^2 e^{-\pi |z|^2} \neq 0 \qquad \text{
  for all } z\in \cd \, .
$$
  If $V_{h_0}f(z) \neq  0$ for all $z\in \rdd $, then the entire
    function $Bf$ does not have zeros in $\cd $ 
and thus is of the form $Bf(z) = e^{h(z)}$ for
some entire function $h$. However, since $V_{h_0}f$ is bounded, it
follows that $|Bf(z)| = e^{\mathrm{Re}\, h(z)} \leq C e^{\pi
  |z|^2/2}$. Thus
$\mathrm{Re}\, h(z) \leq C'+ \pi
  |z|^2/2$. 
By  Carath\'eodory's inequality $h$ is then a quadratic polynomial.  Inversion of the
    Bargmann transform yields the conclusion that
    $f$ must be a generalized Gaussian. 



The connection between zero-free \wid s and  Hudson's theorem now
becomes apparent.  
\begin{theorem} \label{tm-even}
  If $f\in L^2(\R^d) $ and $W(f,f)(x,\xi )  \neq 0$ for all $x,\xi \in
  \bR $, then $f$ is a
  generalized Gaussian. Likewise, if $V_{If}f(x,\xi )  \neq 0$ for all $x,\xi \in
  \bR $, then $f$ is a   generalized Gaussian.
\end{theorem}
\begin{proof}
Assume that  $W(f,f) \neq 0$ everywhere.  Since $W(f,f)$
  is real-valued, this means that either $W(f,f)(z) >0$ for all $z\in
  \bR ^2$ or $W(f,f)(z) <0$ for all $z$, in which case $W(-f,f) >0$
  everywhere. By Theorem~\ref{hud} $f$ must be a
  generalized Gaussian. 
\end{proof}

Next  we fix one of the functions and assume that $g(t) = h_0(t) =
2^{d/4}e^{-\pi t^2}$ is  the normalized  Gaussian in $\rd $. 

\begin{proposition}\label{gwin}
If $V_{h_0}f$ does not have any zero, then $f$ is a generalized Gaussian. 
\end{proposition}
\begin{proof}
Since $|  V_{h_0}f(\bar z)| = | Bf(z)| e^{-\pi |z|^2/2}$ by \eqref{eq:7},
 the entire    function $Bf$ does not have zeros in $\cd $. The last
 part of the proof of Hudson's
    theorem sketched above now implies that $f$ is a generalized
    Gaussian. 
    \end{proof}

The standard polarization identity for the
Wigner distribution now yields the following result on zeros of the real and imaginary parts of the Wigner distribution.
\begin{theorem}\label{wigreal}
  Let $f,g\in \lrd $ and assume that the real part of $W(f,g)$ is
  never zero, $\mathrm{Re}\, W(f,g)(z)\neq 0$ for all $z\in \rdd $.
  Then  $f=a \gamma$ and $g=b \gamma$ for a  generalized
  Gaussian $\gamma $ and $a,b\in \bC $.  Similarly, if
  $\mathrm{Im} W(f,g)(z) \neq 0$ for all $z\in \bR ^2$, then both $f$ and $g$ are generalized Gaussians.
\end{theorem} 

\begin{proof}[Proof]
The polarization identity  states that 
$$
W(f+g,f+g) - W(f-g,f-g) = 4 \mathrm{Re}\, W(f,g) \, .
$$
Thus, if $\mathrm{Re} W(f,g)$ does not have any zeros, then we must have either 
\begin{align}
  W(f+g,f+g)(z) &> W(f-g,f-g)(z)  \label{c89} \\
  \qquad \text{ or } \quad W(f+g,f+g)(z)
& < W(f-g,f-g)(z) \qquad \forall z\in \rdd \, . \notag 
\end{align}
We only treat \eqref{c89}, the second case is obtained by replacing
$g$ by $-g$.

Proceeding as in the proof of Hudson's theorem, we convolve
$W(f\pm g, f\pm g)$ with $W(h_0,h_0)(z) = 2^d e^{-2\pi
  |z|^2}$. 
Since  convolution with a positive function preserves positivity,
\eqref{c89} implies the \emph{strict} pointwise inequality
$$
|V_{h_0}(f+g)|^2 = W(f+g,f+g)\ast W(h_0,h_0)  >  W(f-g,f-g)\ast
W(h_0,h_0) = |V_{h_0}(f-g)|^2 \geq 0 \, .
$$
By Proposition~\ref{gwin} $f+g$ must be a generalized Gaussian $f+g =
\gamma. $

Now, since $W(f-g,f-g) \leq W(\gamma,\gamma)$, Lemma~\ref{hardy} (ii) implies that
$f-g = c\gamma $ for some $c \in \bC $. Solving for $f$ and
$g$ we obtain that $f=\frac{1 + c}{2} \gamma$ and  $g=\frac{1-c}{2} \gamma$,  and therefore both $f$ and $g$ are multiples of
the same generalized Gaussian.   
\end{proof}

\section{Polyanalytic Functions}\label{sec:polyanal}
To find examples of zero-free Wigner distributions, one might fix
a window function $g$ different from a Gaussian and see if  every
\stft\ $V_gf$ possesses at least one zero. We pursue this idea in
dimension $d=1$ for
windows that are finite linear combinations of Hermite functions. This
line of thought connects with the theory of poly-analytic functions
and raises some new questions. 

Let $g = p h_0$ where $p$ is a polynomial and $h_0(t) = 2^{1/4}
e^{-\pi t^2}$ is the normalized Gaussian on $\bR $. We can then write $g$ as a
finite linear combination of Hermite functions, 
\begin{equation}
  \label{eq:3}
g = \sum _{j=0}^N c_j h_j \, ,  
\end{equation}
where $h_j(t) = c_j e^{\pi t^2} \frac{d^j}{dt^j}(e^{-2\pi t^2})$ is the
$j$-th Hermite function with normalization $c_j$, such  that
$\|h_j\|_2=1$. 
It is well known that the \stft\ with respect to such a window
is poly-analytic. To prepare the corresponding formulas,  fix $n$ and
first consider $V_{h_n}f$. This \stft\ can be expressed
with the help of the Bargmann transform of $f$ as follows: 
\begin{equation}
  \label{eq:4}
V_{h_n}f(x,-\xi ) = \frac{1}{\sqrt{\pi ^n n!}} e^{\pi i  x \xi } e^{-\pi
    |z|^2/2} \sum _{k=0}^n \binom{n}{k} (-\pi \bar{z})^{n-k} Bf^{(k)}(z) \, .
\end{equation}
See \cite{BS93,GL09} for some early references of this formula. 
Now assume that the window is of the form $g= \sum _{n=0}^N \sqrt{\pi
  ^n n!} c_n h_n $ and denote the associated polynomial by $P(z) =
\sum _{n=0}^N c_n z^n$. Then 
\begin{align}
  V_gf (\bar{z}) &=  e^{\pi i  x \xi } e^{-\pi
    |z|^2/2} \sum _{n=0}^N c_n \sum _{k=0}^n \binom{n}{k} (-\pi
  \bar{z})^{n-k} Bf^{(k)}(z) \notag \\
&= e^{\pi i  x \xi } e^{-\pi
    |z|^2/2} \sum _{k=0}^N \frac{1}{k!} \Big(\sum _{n=k}^N c_n  \frac{n!}{(n-k)!}(-\pi
  \bar{z})^{n-k} \Big)  Bf^{(k)}(z) \notag \\
&= e^{\pi i x \xi } e^{-\pi
    |z|^2/2}  \sum _{k=0}^N \frac{1}{k!}   Bf^{(k)}(z)
  \overline{P^{(k)}(-\pi z)} \, .   \label{eq:5}
\end{align}
From this explicit formula we see that the function $F(z) =e^{-\pi i x \xi } e^{\pi
    |z|^2/2} V_gf(\bar{z}) $  satisfies the
diffential equation 
\begin{equation}
  \label{eq:6}
  \frac{\partial ^{N+1}}{\partial \bar{z}^{N+1}} F =  0 \, .
\end{equation}
In the established terminology, $F$ is poly-analytic of order
$N+1$ and $F$ is a version of the  \pana\ Bargmann transform of $f$
\cite{abreu10,balkbook,vas00}. Precisely given $g=  \sum _{n=0}^N \sqrt{\pi
  ^n n!} c_n h_n $ with associated polynomial $P(z) = \sum c_n z^n$,
we define the \pana\ Bargmann transform of $f$ with respect to $g$ to
be  
\begin{equation}
  \label{eq:13}
  B_gf(z) = \sum _{k=0}^N \frac{1}{k!}   Bf^{(k)}(z)
  \overline{P^{(k)}(-\pi z)} \, .
\end{equation}
Since $\frac{d^2}{dz d\bar{z}} = \frac{1}{4}\Delta $ is the standard Laplacian,
\eqref{eq:13} can also  be written as
$$
B_gf(z) =   \sum _{k=0}^N \frac{1}{k!} 4^{-k} \Delta ^k \big(Bf(z)
\overline{P(-\pi z)}\big) \, .
$$
The zero sets of  \pana\ functions are   not well understood; in
general they are not discrete and are algebraic curves in
$\bC \simeq \bR ^2$. The zero count is difficult even for such simple
equations as $\bar{z} = P(z)$ for some polynomial $P$~\cite{balkbook,KN08}.

We will use two known, non-trivial results about zeros of \pana\
functions.

\begin{theorem}[Balk]
If $F$ is \pana\ and entire, such that its zero set is contained
in a bounded set, then $F(z) = e^{h(z)}
  Q(z,\bar{z})$ for  some entire function $h$ and some
\pana\ polynomial $Q(z,\bar{z})$. 
\end{theorem}
See \cite{balk66} and \cite{ostro69} for a proof. From this result we obtain the following
  reduction.

  \begin{proposition}
    Fix $g = ph_0$ for some polynomial $p$. If $V_gf $ does not have
    any zeros in $\bC $, then $f(x)=q(x) e^{-(a'x^2+b'x+c')}$ for some
    polynomial $q$ and $a',b',c'\in \bC , \mathrm{Re}\, a'>0$.   
  \end{proposition}
  \begin{proof}
If the zero set of  $V_gf$ is empty, then   the zero set of the
\pana\ Bargmann transform $B_gg$  is also empty (and thus bounded). By
Balk's theorem  $B_gf$ factors into an entire function without zeros and a 
 \pana\ polynomial $Q(z,\bar{z})$, i.e., $B_gf(z) = e^{h(z)} Q(z,\bar{z})$
for $h$ entire. Consequently, by \eqref{eq:5}
we  see that 
\begin{equation}
  \label{eq:iu}
 B_gf(z) = \sum _{k=0}^n \frac{1}{k!} Bf^{(k)}(z)
\overline{P^{(k)}(-\pi z)} = e^{h(z)} Q (z,\bar{z}) \, . 
\end{equation}
 Assume that $P(z) = z^n + \cO (z^{n-1})$ and $Q(z,\bar{z}) = \sum
_{j=0}^m p_j(z) \bar{z}^j$. Comparing the highest term in $\bar{z}$ in
\eqref{eq:iu}, we
conclude that $n=m$ and  that 
$$
(-\pi )^n Bf(z) = e^{h(z)} p_n(z) \, .
$$
Since a function $Bf$ in Fock space grows at most like $e^{\pi
  |z|^2/2}$, we conclude that $\mathrm{Re}\, h(z) \leq c + \pi |z|^2/2$
and therefore  $h(z) $ must be a quadratic polynomial $h(z) =- (az^2 + bz +
c)$ by Caratheodory's theorem. 
On the one hand, we have $Bf(z) = e^{-(az^2+bz+c)}
p_n(z)$, and on the other hand, 
\begin{align}
   Bf(iy) e^{-\pi y^2/2} &= V_{h_0}f(0,-y) = \int _{\bR } f(t) e^{-\pi
                           t^2} e^{2\pi i ty} \d t \notag \\
  &=  \widehat{(f h_0)} (-y) = e^{-ay^2-iby-c}
   p_n(iy) \, .    \label{eq:13b}
\end{align}
Now observe that the inverse Fourier transform of the generalized  Gaussian
$e^{-ay^2-iby-c}$ is again a generalized Gaussian $e^{r(x)}$, and  the polynomial
$p_n(iy)$ turns into the differential operator $p_n (\frac{1}{2\pi }\,
\frac{d}{ dx })$. Consequently, $f h_0 = p_n (\frac{1}{2\pi}\, \frac{d}{ dx
}) e^{r(x)} $ and 
$$
f(x) = q(x) e^{-(a'x^2+b'x+c')}  \, 
$$
for some polynomial $q$. 
   \end{proof}

  In general we cannot say more about the form of $f$. If however, $f=
  q h_0$ with the standard Gaussian $h_0$, then we can reduce the
  problem of zero-free \wid s to a problem of polynomials of two
  variables by means of the \emph{fundamental theorem of algebra for \pana\
  polynomials~\cite{balk68,balkbook}.}
\begin{theorem}[Balk]
  Let $P(z,w)$ be a polynomial in two complex variables. If the exact degree of the  \pana\
polymonial  $P(z,\bar{z})$ as a polynomial of two variables exceeds
the exact degree of $P$ in one of these variables by a factor more
than two, the $P$ has at least one zero.
\end{theorem}
 In other words, let $n_z =
\mathrm{deg}_z\, p$,  $n_{\bar{z}} = \mathrm{deg}_{\bar{z}}\, p$, and $s=
\mathrm{deg}\, p$. If $s > 2 \min (n_z, n_{\bar{z}})$, then $P$ has at
least one zero. See \cite{balk68,balkbook} for the proof. 

From this we deduce the following consequence.

\begin{proposition}\label{prop:lagu}
  Assume that  $g= ph_0$ and $f=qh_0$ for some polynomials $p,q$. If
  $V_gf$ does not have any zeros, then $\mathrm{deg}\, p =
  \mathrm{deg}\, q$. 

In other words, if $f$ and $g$ are finite linear combinations of
Hermite functions, whose highest degrees differ, then the \stft\
$V_gf$ must have a zero. 
\end{proposition}

\begin{proof}
  Under the assumptions stated, $Bf$ is a polynomial $Q$ and the
  \pana\ Bargmann transform is then (for $N$ large enough) 
  \begin{equation}
    \label{eq:13c}
B_gf(z) =   \sum _{k=0}^N \frac{1}{k!}   Q^{(k)}(z)
  \overline{P^{(k)}(-\pi z)} \,     
  \end{equation}
is a \pana\ polynomial. Here $\mathrm{deg}_{z} \,B_gf = \mathrm{deg} \,Q$,
$\mathrm{deg}_{\bar{z}} \, B_gf = \mathrm{deg} \, P$, and
$\mathrm{deg} B_gf\, 
= \mathrm{deg} \, P +  \mathrm{deg} \, Q$. If $\mathrm{deg} \, P \neq
\mathrm{deg} \, Q$, then $\mathrm{deg} \, B_gf > 2 \min (\mathrm{deg}
\, P, \mathrm{deg} \, Q)$, and by the fundamental theorem of algebra $B_gf$
must have a zero. 
\end{proof}

If the degree condition in the fundamental theorem is not satisfied,
then a \pana\ polynomial may not have any zeros, for instance, the
polynomial  $p(z,\bar{z}) =
1+z\bar{z}$ 
does  not have any zeros. 
Note, however, that the \pana\ polynomials arising in
Proposition~\ref{prop:lagu} have a very special structure.

Let us consider the simple example of degree 1 polynomials.  
It is enough to consider
$P(z)=-z+\pi b$ and $Q(z)=z+a$ with $a,b\in \bC $.
Then we are looking for a root of the polynomial 
$$
Q(z)\overline{P(-\pi z)}+Q'(z)\overline{P'(-\pi z)}
=\pi (z+a)(\overline{z}+\overline{b})-1.
$$
Write  $\zeta=\sqrt{\pi}(z+a)$ and  $c=\sqrt{\pi}(b-\overline{a})$,
the equation becomes $\zeta(\overline{\zeta}+c)=1$.  Writing $c$ in polar coordinates $c=\rho e^{i\theta}$
and a further  change of variable  $\zeta=\xi e^{-i\theta}$ yields the
equation $\xi(\overline{\xi}+\rho)=1$.
This equation has at least two real roots, namely $\frac{-\rho\pm\sqrt{\rho^2+4}}{2}$. This proves the following


\begin{proposition}
If $f$ or $g$ is a linear combination of $h_0,h_1$ then $W(f,g)$ has a zero, unless $f,g$ are both multiples of $h_0$.
\end{proposition}

In general, $B_gf$ may be written
as a linear combination of various Laguerre polynomials or complex
Hermite polynomials, see~\cite{folland89} or \cite{Ism16}. Thus $B_gf$
has a very special structure.
One may therefore  hope to  prove the existence of a zero for the
special 
\pana\ polynomials of the form $B_gf$. 
An other simple example is constructed as follows: if $f= p h_0$ for some polynomial $p$ of degree at least 1 and $g = If$ (i.e.,
$g(x) =
f(-x) = p(-x) h_0(x)$), then 
$|V_{If}f (2z)|   = W(f,f)$. The corresponding \pana\ polynomial
$B_{If}f$  then has a zero by Hudson's theorem, although Balk's
fundamental theorem of algebra does not apply.




\section{An example for bounded functions: the short -time Fourier transform with a
  rectangular window}\label{sec:step}

We next study
a class of examples in  which the  decay properties are critical for the
existence or non-existence of 
of zeros of the short-time Fourier transform. For this purpose, we  investigate
the short-time Fourier transform of carefully constructed step
functions. 
We consider the window $g=\un_{(0,1)}$ and   step functions with discontinuities on the union $\bZ
\cup \alpha \bZ $.  We will need a simple lemma first.

\begin{lemma}
\label{lem:step}
Let $0=a_1<a_2<\cdots<a_n<a_{n+1}=1$ and either
$c_1>c_2>\cdots>c_n>0$ or $c_1<c_2 < \dots < c_n$.  Define
$f$ on $(0,1)$ by
$$
f(x)=\sum_{k=1}^n c_k\un_{(a_k,a_{k+1})}.
$$
Then there exists $\xi\in\R$ such that
$$
\int_0^1 f(x)e^{-2\pi i x\xi}\d x=0
$$
if and only if $a_2,\ldots,a_n\in\Q$.
\end{lemma}

\begin{proof} 
We first consider the case when $c_j$ is strictly decreasing.  For
$\xi=0$ we obtain 
$$
\int_0^1 f(x)e^{-2\pi i x\xi}\d x=\int_0^1 f(x)\d x=\sum_{k=1}^n
c_k(a_{k+1}-a_k)>0\, .
$$
For $\xi \neq 0$ we set $I(\xi ) = 2\pi i \xi \int _0^1 f(x) e^{2\pi i
  x\xi } \d x = 2\pi i \xi \hat{f}(-\xi )$. Then 
\begin{align*}
I(\xi) &=2\pi i\xi\int_0^1 f(x)e^{2\pi i x\xi}\d x=\sum_{k=1}^n
         c_k(e^{2\pi i a_{k+1}\xi}-e^{2\pi i  a_{k}\xi})\\ 
&=-c_1+\sum_{k=2}^{n} (c_{k-1}-c_{k})e^{2\pi i a_{k}\xi}+c_ne^{2\pi i \xi}\\
&=c_1\left(\sum_{k=2}^{n} \frac{c_{k-1}-c_{k}}{c_1}e^{2\pi i a_{k}\xi}+\frac{c_n}{c_1}e^{2\pi i \xi}-1\right).
\end{align*}
Thus $\hat{f}(-\xi ) = (2\pi i )^{-1} I(\xi ) = 0$, \fif\ 
\begin{equation}
  \label{eq:c5}
\sum_{k=2}^{n} \frac{c_{k-1}-c_{k}}{c_1}e^{2\pi i
  a_{k}\xi}+\frac{c_n}{c_1}e^{2\pi i \xi} = 1 \, .  
\end{equation}

Since $|e^{2\pi i a_k \xi } | =1$ and 
$$
\sum_{k=2}^{n} \frac{c_{k-1}-c_{k}}{c_1}+\frac{c_n}{c_1}=1 \, 
$$
with positive terms, \eqref{eq:c5} is a convex linear combination of
points in the unit disc of $\C$. 

As $1$ is an extreme point of the unit disc, \eqref{eq:c5} holds 
if and only if it is a convex combination of $1$'s. Therefore, $I(\xi)=0$ if and only if
$e^{2\pi i a_{k}\xi}=1$, $k=2,\ldots,n$, and $e^{2\pi i \xi}=1$. This
implies that  $\xi\in\Z\setminus\{0\}$ and $a_k\xi \in\Z \setminus \{0\}$, $k=2,\ldots,n$.
Thus, if one of the $a_k$'s is irrational, then $\hat{f}(\xi ) \neq 0$
for all $\xi \neq 0$.  
Conversely, if all $a_k\in\Q$, write $a_k=p_k/q_k$, $p_k,q_k\in\Z$ and take $\xi= q_1q_2\cdots q_n$ then $I(\xi)=0$.
Note that this is true whatever the $c_k$'s are.

If $0< c_1 < c_2 < \dots < c_n$, we write
$$
I(\xi )  = c_n \Big( -\frac{c_1}{c_n} \cdot 1 + \sum
_{k=2}^n\frac{c_{k-1}-c_k}{c_n} e^{2\pi i a_k \xi } + e^{2\pi i \xi }
\Big) \, ,$$
and
$I(\xi ) =0$ holds, \fif\
$$
\frac{c_1}{c_n} + \sum_{k=2}^n \frac{c_k- c_{k-1}}{c_n} e^{2\pi i  a_k
  \xi } = -e^{2\pi i \xi } = e^{2\pi i (\xi +1/2)} \, .
$$
Again $\frac{c_1}{c_n} + \sum_{k=2}^n \frac{c_k- c_{k-1}}{c_n} = 1$
with positive terms, and the same convexity argument implies that
$I(\xi ) \neq 0$ whenever one of the $a_k$'s is irrational. 
\end{proof}
The case $c_{k+1}=c_k$  does not need to be considered,
as this amounts to remove $a_{k+1}$. However, 
the proof of Lemma \ref{lem:step} no longer  works when the
sequence $c_k$ is not monotonic. Instead we prove  the following.  

\begin{lemma}
\label{lem:nonmono}
Let $\alpha\in(0,1)\setminus\Q$, $0<a<1-\alpha$, $0<b,c<d$ 
and define
$$
f_a(t)=b\un_{(0,a)}+d\un_{(a,a+\alpha)}+c\un_{(a+\alpha,1)}.
$$
Then there exists  $ a\in (0, 1-\alpha )$ and $\xi\in\R$, such that $\hat{f_a}(\xi) =0$. 
\end{lemma}

\begin{proof} 
  We may assume without loss of
  generality that $d=1$ and $0<b,c<1$. Since $f_a(t) >0$ on $(0,1)$,
  clearly $\hat{f}(0)  = \int_0^1 f(t)\d t>0$. 
For $\xi\not=0$ we use  again
$I(\xi)=2\pi i\xi \hat{f}(-\xi )$. 
The computation shows that
\begin{align*}
  I(\xi) &=b\bigl(e^{2\pi i a\xi}-1\bigr)+\bigl(e^{2\pi i (a+\alpha)\xi}-e^{2\pi i a\xi}\bigr)
 +c\bigl(e^{2\pi i \xi}-e^{2\pi i(a+\alpha)\xi}\bigr)\\
 &=e^{2\pi i a\xi}\bigl((1-c)e^{2\pi i \alpha\xi}-(1-b)\bigr)+ce^{2\pi i \xi}-b.
\end{align*}

If $I(\xi)=0$, then 
\begin{equation}
  \label{eq:ba1}
\bigl|(1-c)e^{-2\pi i \alpha\xi}-(1-b)\bigr|^2=\bigl|ce^{-2\pi i \xi}-b\bigr|^2  
\end{equation}
or
$(1-c)^2+(1-b)^2-2(1-c)(1-b)\cos 2\pi\alpha\xi=c^2+b^2-2bc\cos2\pi\xi
\, .$
We obtain that  $I(\xi ) = 0$  implies the identity
\begin{equation}
\label{eq:almostperiod}
1-b-c=(1-b-c+bc)\cos 2\pi\alpha\xi-bc\cos2\pi\xi.
\end{equation}

Let $M=(1-b)(1-c)+bc = 1-b-c+2bc$. Since $M\pm (1-b-c) =
\{2(1-b)(1-c), 2bc\}$, we have $|1-b-c| < M$. 
Consider the function $\psi (\xi ) =  (1-b-c+bc)\cos
2\pi\alpha\xi-bc\cos2\pi\xi$. Clearly, $|\psi (\xi ) | \leq M$, and
more importantly, since  $\alpha\notin\Q$, 
$\psi $ is almost periodic and  thus takes every  value in $(-M,M)$
infinitely often. More precisely,
for each $t\in(-M,-M)$ there exists a uniformly discrete set $\Lambda_t$ such that, for every
$\xi\in\Lambda_t$, $(1-b-c+bc)\cos
2\pi\alpha\xi-bc\cos2\pi\xi=t$. See, e.g., \cite{Besbook}. 
Let  $\Lambda :=  \Lambda_{1-b-c}= \{ \xi \in \bR : (1-b-c+bc)\cos
2\pi\alpha\xi-bc\cos2\pi\xi=1-b-c\} $.

Exploiting \eqref{eq:ba1} further, we see that $I(\xi ) = 0, \xi \neq 0,$ implies
that either

(i)$
\frac{b-ce^{2\pi i \xi}}{(1-c)e^{2\pi i \alpha\xi}-(1-b)} = e^{i\theta }
$
for some $\theta \in \bR $, or

(ii) $(1-c)e^{2\pi i \alpha\xi}-(1-b) = b-ce^{2\pi i \xi} = 0$. This
identity is clearly  impossible when $b\neq c$. If $b=c$, then
$e^{2\pi i \alpha \xi }-1=e^{2\pi i \xi }-1=0$, which again cannot
happen because $\alpha \not \in \bQ$.

The remaining alternative (i) shows that $I(\xi )=0$ for $\xi \neq 0$, if and only
if
$$
\frac{b-ce^{2\pi i \xi}}{(1-c)e^{2\pi i \alpha\xi}-(1-b)} = e^{i\theta
} = e^{2\pi i a\xi }
\, .
$$
Given $\theta \in (0,2\pi )$ we choose $\xi \in \Lambda $ so large
that $a =\theta /\xi \in (0,1-\alpha )$. Thus we have shown that there
exist $a\in (0,1-\alpha )$ and $\xi \in \bR $ such that $\hat{f_a}(\xi
) = 0$. 

\end{proof}

The following result shows that construction of  a
zero-free \wid s  with  $g=\un _{(0,1)}$ might be rather delicate. We assume that $f$ is a
step function with jumps on $\bZ \cup \alpha \bZ $, then every \stft\
$V_gf$ with $f\in L^2(\bR )$  possesses a zero.  However, if we ``slightly'' enlarge the function space and
also consider bounded step functions,  then we can produce a zero-free
\stft\ $V_gf$. 

\begin{theorem}
\label{cor:counterex}
Let $(c_k)_{k\in\Z}\subset\R$ be a bounded positive sequence. Let $\alpha\in(0,1)\setminus\Q$ and define
the sequence $(a_k)$ by $a_{2k}=k$, $a_{2k+1}=k+\alpha$.
Let $\chi=\un_{(0,1)}$ and
\begin{equation}
  \label{eq:ab2}
\ffi=\sum_{k\in\Z} c_k\un_{(a_k,a_{k+1})}.  
\end{equation}

(i)  If $(c_k)_{k\in\Z}$ is monotonic, then $\ffi\in L^\infty(\R)$ and
$V_\chi f$ never vanishes. \\

(ii) If $\ffi\in L^p(\R)$ for $1\leq p<\infty $, then there exists $(x,\xi)\in\R^2$ such that $V_\chi f(x,\xi)=0$.
\end{theorem}

\begin{proof}
Clearly, for $1\leq p\leq \infty$,  $\ffi \in L^p(\bR )$, if and only if  $(c_k)\in \ell ^p(\bZ
)$.  We write
$$
\ffi=\sum_{k\in\Z} \bigl(c_{2k}\un_{(k,k+\alpha)}
+c_{2k+1}\un_{(k+\alpha,k+1)}\bigr) \, .
$$
Then the  function $\ffi (t) \chi (t-x)$ is a step function consisting of  three
terms. To describe the correct terms, we write  $x=j+u$ with $j\in\Z$
and $u\in[0,1)$ and distinguish two cases.

\noindent{\bf Case 1.}  If $u\in [0,\alpha)$, then 
\begin{multline}
  V_\chi \ffi(x,\xi)\\
	= \int _\bR \Big( c_{2j} \un _{(j+u,j+\alpha)} (t) + c_{2j+1} \un _{(j+\alpha,j+1)}(t) +
      c_{2j+2} \un _{(j+1,j+1+u)}(t)\Big) \, e^{-2\pi i  t \xi } \d t
                       \\
    = e^{-2\pi i (j+u)\xi } \, \int_0^1\bigl(c_{2j}\un_{(0,\alpha-u)}(t)+c_{2j+1}\un_{(\alpha-u,1-u)}(t)
+c_{2j+2}\un_{(1-u,1)}(t)\bigr)e^{-2\pi i t\xi}\d t. \label{eq:u0alpha}
\end{multline}
We observe right away that at least one of the endpoints $\alpha -u$
or $1-u$ must be irrational, since $\alpha  \not \in \bQ $. 

\noindent{\bf Case 2.} If  $u\in[\alpha,1)$, then
\begin{multline}
  V_\chi \ffi(x,\xi)  \\
	= \int _\bR \Big( c_{2j+1} \un _{(j+u,j+1)}(t) +
                       c_{2j+2} \un _{(j+1, j+1+\alpha )}(t) +
                       c_{2j+2} \un_{(j+1+\alpha , j+1+u)}(t) \big) \,
                       e^{-2\pi i \xi t} \d t \\ 
= e^{-2\pi i \xi (j+u)}\, \int_0^1\bigl(c_{2j+1}\un_{(0,1-u)}(t)+c_{2j+2}\un_{(1-u,1-u+\alpha)}(t)
+c_{2j+3}\un_{(1-u+\alpha,1)}(t)\bigr)e^{-2\pi i t\xi}\d
t. \label{eq:ualpha1} 
\end{multline}
Again, at least one of the endpoints 
$1-u$ and $1-u+\alpha$ must be irrational, because $\alpha \not \in
\bQ $.

Proof of (i):  If the coefficient sequence  $(c_j)$ is monotonic, then
 Lemma \ref{lem:step} is applicable, and 
 we conclude that $V_\chi \ffi(x,\xi)\not=0$ for all  $\xi\in \bR $
 and all $x\in \bR$. 

 Proof of (ii): Assume now that $\ffi\in L^p(\R)$ for $p< \infty$,
 then $(c_j) \in \ell ^p$ cannot be monotonic.  Thus there exists
 $j_0$ such that either  $c_{2j_0},c_{2j_0+2}<c_{2j_0+1}$ or
 $c_{2j_0+1},c_{2j_0+3}<c_{2j_0+2}$.
 In the former case we choose $x=j_0 + u$ for $u\in (0,\alpha )$ (to be
 determined in a moment) and we work with \eqref{eq:u0alpha}.  In the latter case we choose $x=j_0 + u$ for
 $u\in (\alpha,1 )$ and work with \eqref{eq:ualpha1}. In both cases
 the length of the middle interval $(\alpha -u, 1-u)$ and
 $(1-u,1-u+\alpha)$ is irrational. 
According to Lemma \ref{lem:nonmono}, there exists suitable $u$ and
$\xi \in \bR $ such that
$V_\chi \ffi(j_0+u,\xi)=0$.
\end{proof}

\subsection*{Acknowledgement} The authors would like to thank the Erwin
Schr\"odinger Institute for its  hospitality and
great conditions during the work on this paper. Special thanks go to
Martin Ehler for his help with Mathematica and his helpful
discussions. We are also grateful to Andrii Bondarenko, who showed Example \ref{ex:Gamma} to us.

P.J. acknowledges that this study has been carried out with financial support from the French State, managed
by the French National Research Agency (ANR) in the frame of the ``Investments for
the future'' Programme IdEx Bordeaux - CPU (ANR-10-IDEX-03-02).

E.M. is supported  by  Project 275113
of the Research Council of Norway and NSF grant no. DMS-1638352

\def\cprime{$'$}


\begin{thebibliography}{10}

\bibitem{abreu10}
L.~D. Abreu.
\newblock Sampling and interpolation in {B}argmann-{F}ock spaces of
  polyanalytic functions.
\newblock {\em Appl. Comput. Harmon. Anal.}, 29(3):287--302, 2010.

\bibitem{balk66}
M.~B. Balk.
\newblock Entire poly-analytic functions with a bounded set of zeros.
\newblock {\em Izv. Akad. Nauk Armjan. SSR Ser. Mat.}, 1(5):341--357, 1966.

\bibitem{balk68}
M.~B. Balk.
\newblock The fundamental theorem of algebra for polyanalytic polynomials.
\newblock {\em Litovsk. Mat. Sb.}, 8:401--404, 1968.

\bibitem{balkbook}
M.~B. Balk.
\newblock {\em Polyanalytic functions}, volume~63 of {\em Mathematical
  Research}.
\newblock Akademie-Verlag, Berlin, 1991.

\bibitem{bayer10}
D.~Bayer.
\newblock {\em Bilinear time-frequency distributions and pseudodifferential
  operators}.
\newblock PhD thesis, University of Vienna, 2010.

\bibitem{BG15}
D.~Bayer and K.~Gr\"ochenig.
\newblock Time-frequency localization operators and a {B}erezin transform.
\newblock {\em Integral Equations Operator Theory}, 82(1):95--117, 2015.

\bibitem{Besbook}
A.S. Besicovitch.
\newblock Almost periodic functions.
\newblock Cambridge, 1932.

\bibitem{BS93}
S.~Brekke and K.~Seip.
\newblock Density theorems for sampling and interpolation in the
  {B}argmann-{F}ock space. {III}.
\newblock {\em Math. Scand.}, 73(1):112--126, 1993.

\bibitem{Goss11}
M.~A. de~Gosson.
\newblock {\em Symplectic methods in harmonic analysis and in mathematical
  physics}, volume~7 of {\em Pseudo-Differential Operators. Theory and
  Applications}.
\newblock Birkh\"auser/Springer Basel AG, Basel, 2011.

\bibitem{folland89}
G.~B. Folland.
\newblock {\em Harmonic Analysis in Phase Space}.
\newblock Princeton Univ. Press, Princeton, NJ, 1989.

\bibitem{book}
K.~Gr{\"o}chenig.
\newblock {\em Foundations of time-frequency analysis}.
\newblock Birkh\"auser Boston Inc., Boston, MA, 2001.

\bibitem{GL09}
K.~Gr{\"o}chenig and Y.~Lyubarskii.
\newblock Gabor (super)frames with {H}ermite functions.
\newblock {\em Math. Ann.}, 345(2):267--286, 2009.

\bibitem{GZ01}
K.~Gr{\"o}chenig and G.~Zimmermann.
\newblock Hardy's theorem and the short-time {F}ourier transform of {S}chwartz
  functions.
\newblock {\em J. London Math. Soc.}, 63:205--214, 2001.

\bibitem{hr2}
E.~Hewitt and K.~A. Ross.
\newblock {\em Abstract harmonic analysis. {V}ol. {I}{I}: {S}tructure and
  analysis for compact groups. {A}nalysis on locally compact {A}belian groups}.
\newblock Springer-Verlag, New York, 1970.
\newblock Die Grundlehren der mathematischen Wissenschaften, Band 152.

\bibitem{hr1}
E.~Hewitt and K.~A. Ross.
\newblock {\em Abstract harmonic analysis. {V}ol. {I}}.
\newblock Springer-Verlag, Berlin, second edition, 1979.
\newblock Structure of topological groups, integration theory, group
  representations.

\bibitem{hudson74}
R.~L. Hudson.
\newblock When is the {W}igner quasi-probability density non-negative?
\newblock {\em Rep. Mathematical Phys.}, 6(2):249--252, 1974.

\bibitem{Gant}
F.~R. Gantmacher.
\newblock {\it The Theory of Matrices, vol. II.}
\newblock Chelsea Publ., New York, 1959.
\bibitem{Ism16}
M.~E.~H. Ismail.
\newblock Analytic properties of complex {H}ermite polynomials.
\newblock {\em Trans. Amer. Math. Soc.}, 368(2):1189--1210, 2016.

\bibitem{janssen84}
A.~J. E.~M. Janssen.
\newblock A note on {H}udson's theorem about functions with nonnegative {W}igner
  distributions.
\newblock {\em SIAM J. Math. Anal.}, 15(1):170--176, 1984.

\bibitem{MR86m:81054}
A.~J. E.~M. Janssen.
\newblock Bilinear phase-plane distribution functions and positivity.
\newblock {\em J. Math. Phys.}, 26(8):1986--1994, 1985.


\bibitem{janssen96}
A.~J. E.~M. Janssen.
\newblock Some {W}eyl-{H}eisenberg frame bound calculations.
\newblock {\em Indag. Math.}, 7:165--182, 1996.

\bibitem{KV08}
O.~M. Katkova and A.~M. Vishnyakova.
\newblock A sufficient condition for a polynomial to be stable.
\newblock {\em J. Math. Anal. Appl.}, 347(1):81--89, 2008.

\bibitem{macdonald99}
H.~M. MacDonald.
\newblock Zeroes of the {B}essel {F}unctions.
\newblock {\em Proc. Lond. Math. Soc.}, 30:165--179, 1898/99.

\bibitem{encycl}
Macdonald function.
\newblock Encyclopedia of Mathematics. \\
\newblock URL: http://www.encyclopediaofmath.org.

\bibitem{KN08}
D.~Khavinson and G.~Neumann.
\newblock From the fundamental theorem of algebra to astrophysics: a
  ``harmonious'' path.
\newblock {\em Notices Amer. Math. Soc.}, 55(6):666--675, 2008.

\bibitem{KLSW12}
J.~Kiukas, P.~Lahti, J.~Schultz, and R.~F. Werner.
\newblock Characterization of informational completeness for covariant phase
  space observables.
\newblock {\em J. Math. Phys.}, 53(10):102103, 11, 2012.

\bibitem{lieb90}
E.~H. Lieb.
\newblock Integral bounds for radar ambiguity functions and {W}igner
  distributions.
\newblock {\em J. Math. Phys.}, 31(3):594--599, 1990.

\bibitem{LS18}
F.~Luef and E.~Skrettingland.
\newblock Convolutions for {B}erezin quantization and {B}erezin-{L}ieb
  inequalities.
\newblock {\em J. Math. Phys.}, 59(2):023502, 11, 2018.

\bibitem{ostro69}
I.~V. Ostrovs\cprime~ki\u\i.
\newblock A theorem of {M}. {B}. {B}alk.
\newblock {\em Mat. Fiz. i Funkcional. Anal.}, (Vyp. 1):191--202, 263, 1969.


\bibitem{sch51}
I.~J. Schoenberg.
\newblock On {P}\'olya frequency functions. {I}. {T}he totally positive
  functions and their {L}aplace transforms.
\newblock {\em J. Analyse Math.}, 1:331--374, 1951.


\bibitem{Str}
S. Strelitz. 
\newblock On the Routh-Hurwitz Problem.
\newblock {\em Amer. Math. Monthly} 84, 542-544, 1977.

\bibitem{vas00}
N.~L. Vasilevski.
\newblock Poly-{F}ock spaces.
\newblock In {\em Differential operators and related topics, {V}ol. {I}
  ({O}dessa, 1997)}, volume 117 of {\em Oper. Theory Adv. Appl.}, pages
  371--386. Birkh\"auser, Basel, 2000.

  \bibitem{watson95}
G.~N. Watson.
\newblock {\em A treatise on the theory of {B}essel functions}.
\newblock Cambridge Mathematical Library. Cambridge University Press,
  Cambridge, 1995.
\newblock Reprint of the second (1944) edition.

\end{thebibliography}

\end{document}